\documentclass[letterpaper]{article}

\usepackage{amsmath}
\usepackage{amsfonts}
\usepackage{amsthm}
\usepackage{amssymb}
\usepackage{datetime}
\usepackage{color}
\usepackage{url}

\newtheorem{theorem}{Theorem}
\newtheorem{lemma}[theorem]{Lemma}
\newtheorem*{lemma*}{Lemma}

\newtheorem*{theorem*}{Theorem}
\theoremstyle{definition}

\newtheorem{remark}[theorem]{Remark}
\newtheorem*{remark*}{Remark}

\newcommand{\all}{\hbox{for all}}

\newcommand{\bra}[2]{\langle#1,#2\rangle}

\newcommand{\dand}{\hbox{ and }}

\newcommand{\half}{{\textstyle\frac{1}{2}}}

\newcommand{\lr}{\Longrightarrow}

\newcommand{\minn}{\min\nolimits}

\newcommand{\R}{\cal R}
\newcommand{\RR}{\mathbb R}

\renewcommand{\S}{\cal S}

\newcommand{\supn}{\sup\nolimits}

\newcommand{\ts}{\textstyle}
\newcommand{\tsum}{\textstyle\sum}

\newcommand{\Lem}{Lemma~\ref}

\newcommand{\Sec}{Section~\ref}

\newcommand{\Thm}{Theorem~\ref}
\newcommand{\Thms}{Theorems~\ref}

\title{$m$th roots of the identity operator and the geometry conjecture}
\author{
Stephen Simons
\thanks{
Department of Mathematics, University of California, Santa Barbara, CA\ 93106-3080, U.S.A.
Email: \texttt{stesim38@gmail.com}.}}
\date{}

\begin{document}
\maketitle
\begin{abstract}
\noindent In this paper, we give three different new proofs of the validity of the\break {\em geometry conjecture} about cycles of projections onto nonempty closed, convex subsets of a Hilbert space.   The first uses a simple minimax\break theorem, which depends on the finite dimensional Hahn-Banach theorem.   The second uses Fan's inequality, which has found many applications in\break optimization and mathematical economics.   The third uses three results on maximally monotone operators on a Hilbert space. 
\end{abstract}



\section{Introduction}\label{Introduction}
Let $H$ be a Hilbert space and $m \ge 2$.   For $j = 1, \dots, m$, let $C_j$ be a nonempty closed convex subset of $H$ and $P_j$ be the projection of $H$ onto $C_j$.   The element $z = (z_1,z_2,\dots,z_m)$ of $H^m$ is said to be a  {\em cycle} if
\begin{equation*}
z_1 = P_1z_m,\ z_2 = P_2z_1,\ \dots,\ z_m = P_mz_{m - 1}.
\end{equation*}
The {\em geometry conjecture}, which was first formulated in \cite[Conjecture 5.1.6]{BB} in 1997, stated that {\em there exist $v_1,\dots v_m \in H$ such that $v_1 + \cdots v_m = 0$ and}
\begin{equation}\label{w}
\{z_m: (z_1,\dots z_m)\hbox{ is a cycle}\} =   C_{m} \cap
(C_{m - 1} + v_{m - 1}) \cap
(C_{1} + \tsum_{i = 1}^{m - 1}v_i).
\end{equation}
This conjecture was finally solved in the affirmative in \cite[Theorem 9, pp.\ 7--8]{ABRW}.   In this paper, we give a proof of this conjecture which is simpler than that in \cite{ABRW}, and extends the result to a more general situation.   (See \Thm{Pthm}.)
\par
\Sec{ALGsec} contains the algebraic results on which our analysis is based.   It discusses certain linear operators on a vector space, $X$, which has no additional structure.
\par
In \Sec{HILsec}, we assume that $X$ is endowed with a Hilbert space structure, and we assume that the operator $R$ introduced in \Sec{ALGsec} is an isometry.   All the subsequent analysis depends on the disarmingly simple Equations \eqref{3Norms} and \eqref{m2AB}, which will be used 6 times in what follows.    It is important to realize that the Hilbert space $X$ is not necessarily the same as the Hilbert space, $H$, discussed at the beginning of this introduction --- for the material in \Sec{Cycles}, $X$ will be $H^m$.
\par
The first main result in \Sec{HILsec} is the somewhat technical \Lem{mLem}, which is about the support function of a nonempty closed, convex subset of $X$.   Our proof of this goes by way of a simple minimax theorem, which depends on nothing more complicated than the Hahn-Banach theorem in finite dimensional spaces.  We state the relevant minimax theorem in \Thm{MMT}.  We do not work in the Hilbert space $X$, but the Hilbert subspace $Y$ introduced in \Sec{ALGsec}.   \Thm{Pthm} is the main result of the paper, and gives a generalization to this more abstract situation of the solution to the geometry conjecture.
\par
In \Sec{Cycles}, we apply \Thm{Pthm} to obtain a solution to the geometry conjecture.   (See \Thms{vThm} and \ref{Geometric}.)
\par
In \Sec{Other} we give two other ways of obtaining \Thm{Pthm}, the {\em  Fan's inequality approach} and the {\em maximally monotone operator approach}.   
\par
In \Thm{FanThm}, we state Fan's inequality, and show how this can be used instead of \Thm{MMT} to obtain \Lem{mLem}.   This is the {\em  Fan's inequality approach}.
\par
In \Lem{eLem}, we give a result very similar to \Lem{mLem} using three (non--trivial) results on maximally monotone operators on a Hilbert space. This corresponds more closely with the method used in \cite{ABRW}, and is what we call the {\em maximally monotone operator approach}.
\par
The author would like to express his sincere thanks to Heinz Bauschke and Xianfu Wang for reading through the first version of this paper, and making suggestions that have considerably improved the exposition. 
\par
All spaces in this paper will be {\em real}.
\section{The average operator}\label{ALGsec}
Let $X$ be a linear space, $R\colon\ X \to X$ be linear and $R^m = I_X$.      Let
\begin{equation*}
A := \ts\frac1m\sum_{i = 1}^{m}R^{i}. \hbox{   ($A$ stands for {\em average}.)}
\end{equation*}  Let $Y$ be the subspace $\{y \in X\colon\  Ay = 0\}$ of $X$.   We now define $S\colon\ X \to X$ by $S:= R - I_X$ and $Q\colon\ Y \to X$ by $Q:= \frac1m\tsum_{i = 1}^{m - 1}iR^i$.   (Compare with\break \cite[Equation (28), p.\ 4]{ABRW}.)        
\begin{lemma}\label{Alem}
We have
\begin{equation}\label{SXY}
S(X) \subset Y,
\end{equation}
\begin{equation}\label{SB}
\all\ y \in Y,\ Qy \in Y \dand S(Qy) = y.
\end{equation}
\end{lemma}
\begin{proof}
We have 
\begin{equation*}
mAS = \tsum_{i = 1}^{m}(R^{i + 1} - R^i) = R^{m + 1} - R = 0
\end{equation*}
and \eqref{SXY} follows immediately.
\par
Note that
\begin{equation*}
mAR = mRA =  \tsum_{i = 1}^{m}R^{i + 1} = \ts\sum_{i = 2}^{m}R^{i} + R^{m + 1} = \ts\sum_{i = 2}^{m}R^{i} + R = mA.
\end{equation*}
Thus $AR = RA = A$. It is easy to deduce from this that, if $1 \le k \le m$ then
\begin{equation}\label{A1}
AR^k = A\hbox{ and }R^kA = A.
\end{equation}
Let $y \in Y$.  Then $A(Qy) = (AQ)y = \frac1m\tsum_{i = 1}^{m - 1}iAR^iy$.   From \eqref{A1},  $A(Qy) = \frac1m\tsum_{i = 1}^{m - 1}iAy = 0$, and so $Qy \in Y$.   Now
\begin{equation*}
mRQy = \tsum_{i = 1}^{m - 1}iR^{i + 1}y = \tsum_{i = 2}^{m}(i - 1)R^{i}y = \tsum_{i = 2}^{m - 1}(i - 1)R^{i}y + my - y
\end{equation*}
and
\begin{equation*}
mQy = \tsum_{i = 1}^{m - 1}iR^iy = Ry + \tsum_{i = 2}^{m - 1}iR^iy.
\end{equation*}
Thus  
\begin{align*}
mS(Qy) &= mRQy - mQy = \tsum_{i = 2}^{m - 1}(i - 1)R^{i}y + my- y -  Ry - \tsum_{i = 2}^{m - 1}iR^iy\\
&= my - y - Ry - \tsum_{i = 2}^{m - 1}R^iy = my - mAy = my.
\end{align*}
The proof of \eqref{SB} is now completed by dividing by $m$.
\end{proof}
\begin{remark}
Using the same argument as in \eqref{SB}, we can prove that, for all $y \in Y$, $QSy = y$.
If we define $S_0 := S|_Y$ then it follows that $S_0$ and $Q$ are bijections from $Y$ onto $Y$ and $S_0^{-1} = Q$. 
\par
There are actually stronger versions of some of the relationships\break established above.   \eqref{SB} can be strengthened to $SQ = QS =  I_X - A$. (Compare\break \cite[Proposition 3, p.\ 4]{ABRW}.)   It can also be proved that $AQ = QA =  \ts\frac{m - 1}{2}A$.  (Compare \cite[Proposition 2, p.\ 4]{ABRW}.)   Finally, $A$ is a {\em projection}, that is to say, $A^2 = A$ and, further, $SA = AS = 0$.
\par
 The material in this section is related to the recent paper \cite{ABRW2}.
\end{remark}
\section{Results for Hilbert spaces}\label{HILsec}
We now suppose that $X$ is a Hilbert space.   We have the following Theorem with two parts and a one--line proof.  \Thm{OneLine}(a) is due to Xianfu Wang.    
\begin{theorem}\label{OneLine}
Let ${\R}\colon X \to X$ be a linear map and ${\S}:= {\R} - I_X$.   Then:
\par\noindent
{\rm(a)}\enspace ${\R}$ is nonexpansive if, and only if, for all $x \in X$, $\|{\S}x\|^2 \le -2\bra{{\S}x}{x}$.
\par\noindent
{\rm(b)}\enspace ${\R}$ is an isometry if, and only if, for all $x \in X$, $\|{\S}x\|^2 = -2\bra{{\S}x}{x}$.
\end{theorem}
\begin{proof}
Here is the one line: for all $x \in X$, 
\begin{equation*}
\|{\S}x\|^2 + 2\bra{{\S}x}{x} = \bra{{\S}x}{{\S}x + 2x} = \bra{{\R}x - x}{{\R}x + x} = \|{\R}x\|^2 - \|x\|^2.
\end{equation*}
That completes the proof.
\end{proof}
We now suppose that $R$, $Y$, $S$ and $Q$ are as in \Sec{ALGsec}, and that $R$ is a {\em linear isometry}.  Obviously, $Y$ is closed, so $Y$ is a Hilbert space in its own right.  Furthermore, $S$ is continuous on $X$ and $Q$ is continuous on $Y$.   From \Thm{OneLine}(b),
\begin{equation}\label{3Norms}
\all\ x \in X,\ \|Sx\|^2 = -2\bra{Sx}{x}.
\end{equation}
Now let $y \in Y$.  Setting $x = Qy$ in \eqref{3Norms} and using \eqref{SB}, $\|y\|^2 = -2\bra{y}{Qy}$.   Consequently,
\begin{equation}\label{m2AB}
\ts \bra{Qy}{y} = -\half\|y\|^2.
\end{equation}
In what follows, we suppose that $C$ is a nonempty, closed convex subset of $X$ and $\sigma_C$ is the support function of $C$, defined by $\sigma_C(t) := \sup\bra{C}{t}$.  Let
\begin{equation*}
D := \{y \in Y\colon\ \sigma_C(y) + \half\|y\|^2 \le 0\}.
\end{equation*}
We now give some properties of the set $D$.
\begin{lemma}\label{Dlem}
The set $D$ is nonempty, convex and weakly compact. 
\end{lemma}
\begin{proof}
Since $D \ni 0$, $D$ is nonempty.   Since $\sigma_C + \half \|\cdot\|^2$ is proper, convex and lower semicontinuous on $Y$, $D$ is convex and closed, hence weakly closed.   Fix $c_0 \in C$.   Then
\begin{equation*}
d \in D \lr \bra{c_0}{d} + \half\|d\|^2 \le 0 \lr -\|c_0\|\|d\| + \half\|d\|^2 \le 0 \lr \|d\| \le 2\|c_0\|,
\end{equation*}
so $D$ is bounded and, consequently, weakly compact.
\end{proof}  
We now state a simple minimax theorem that depends only on the finite\break dimensional Hahn-Banach theorem. See \cite[Theorem 3.2, p.\ 25]{HBM} and\break \cite[Theorem 3.1, p.\ 17]{MANDM}.   The proof in \cite{MANDM} is somewhat more direct.   There are many more minimax theorems - there is a survey of these in \cite{SURVEY}.
\begin{theorem}\label{MMT}
Let $B$ be a nonempty convex subset of a vector space, $D$ be a nonempty convex subset of a vector space and $D$ also be a compact Hausdorff topological space.   Let $f\colon\ B \times D \to \RR$ be concave with respect to its first variable, and convex and lower semicontinuous with respect to its second variable.   Then
\begin{equation}
\supn_{b \in B}\minn_{d \in D}f(b,d) = \minn_{d \in D}\supn_{b \in B}f(b,d).
\end{equation}
\end{theorem}
\par
\begin{lemma}\label{mLem}
There exist $d,e \in Y$ such that
\begin{equation}\label{mSe}
d = Se \in D
\end{equation}
and
\begin{equation}\label{mz}
\all\ x \in S^{-1}D,\ \sigma_{C}(Se) + \bra{Sx -  Se}{e} - \sigma_{C}(Sx) \le 0.
\end{equation}
\end{lemma}
\begin{proof}
Define $f\colon\ D \times D \to \RR$ by
\begin{equation*}
f(b,d) := \sigma_{C}(d) + \half\|d\|^2 + \bra{b}{Qd} -  \sigma_{C}(b).
\end{equation*}
Then $f$ is concave with respect to its first variable, and convex and lower semicontinuous with respect to its second variable.   Let $b$ be an arbitrary element of $D$.   Then, from \eqref{m2AB},
\begin{align}
\minn_{d \in D}f(b,d) \le f(b,b) &= \sigma_{C}(b) + \half\|b\|^2 + \bra{b}{Qb} -  \sigma_{C}(b)\label{bb1}\\
&= \half\|b\|^2 + \bra{b}{Qb} = 0.\label{bb2}
\end{align}
Thus $\supn_{b \in D}\minn_{d \in D}f(b,d) \le 0$.   From \Thm{MMT}, $\minn_{d \in D}\supn_{b \in D}f(b,d) \le 0$.   Consequently, there exists $d \in D$ such that,
\begin{equation*} 
\supn_{b \in D}[\sigma_{C}(d) + \half\|d\|^2 + \bra{b}{Qd} - \sigma_{C}(b)] \le  0.
\end{equation*}
From \eqref{m2AB}, $\half\|d\|^2 = -\bra{Qd}{d}$, and so we can rephrase the above as
\begin{equation*} 
\supn_{b \in D}[\sigma_{C}(d) + \bra{b - d}{Qd} - \sigma_{C}(b)] \le  0.
\end{equation*}
Let $e := Qd \in Y$.  From \eqref{SB},  $Se = d \in D$, so \eqref{mSe} is satisfied and,
\begin{equation*} 
\supn_{b \in D}[\sigma_{C}(Se) + \bra{b - Se}{e} -  \sigma_{C}(b)] \le 0.
\end{equation*}
\eqref{mz} follows easily from this.
\end{proof}
We now come to the main result of this section.
\begin{theorem}\label{Pthm}
Let $d,e$ be as in \Lem{mLem}, $P_C$ be the projection of $X$ onto $C$ and $z \in C$. Then the following conditions are equivalent:
\begin{gather}
z = P_CRz\label{P1}\\
z \in S^{-1}D\label{P3}\\
z \in S^{-1}d.\label{P5}
\end{gather}
\end{theorem} 
\begin{proof}
(\eqref{P1}$\iff$\eqref{P3})\enspace It is a classical result in Hilbert space theory that \eqref{P1} is equivalent to the statement ``for all $c \in C$, $\bra{Rz - z}{c - z} \le 0$'', which is obviously equivalent to
\begin{equation}\label{P2}
\sigma_C(Sz) \le \bra{Sz}{z}.
\end{equation}
From \eqref{3Norms}, $\bra{Sz}{z} = -\half\|Sz\|^2$, so \eqref{P2} is equivalent to: ``$\sigma_C(Sz) \le -\half\|Sz\|^2$'',\quad i.e.,\quad``$\sigma_C(Sz) + \half\|Sz\|^2 \le 0$'',\quad which is exactly \eqref{P3}.
   
(\eqref{P3}$\iff$\eqref{P5}) If \eqref{P3} is true then, from the above equivalences, \eqref{P2} is true.   From \eqref{P2}, \eqref{P3} and \eqref{mz},
\begin{equation*}
\sigma_{C}(Se) + \bra{Sz -  Se}{e} - \bra{Sz}{z} \le \sigma_{C}(Se) + \bra{Sz -  Se}{e} - \sigma_{C}(Sz) \le 0.
\end{equation*}
Since $z \in C$, $\bra{Se}{z} \le \sigma_{C}(Se)$.   Consequently,
\begin{equation*}
\bra{Se}{z} + \bra{Sz -  Se}{e} - \bra{Sz}{z} \le 0,
\end{equation*}
from which\quad $\bra{S(z -  e)}{z - e} \ge 0$.\quad   It now follows from \eqref{3Norms} that\quad $\|S(z - e)\| = 0$, and so\quad $Sz = Se = d$,\quad giving \eqref{P5}.   Finally, it is obvious that \eqref{P5}$\lr$\eqref{P3}.
\end{proof}
%
\section{Cycles of projections}\label{Cycles}
Let $X := H^m$, $C := C_1 \times \cdots \times C_m \subset H^m$,
\begin{equation*}
R(x_1,x_2, \dots, x_m) := (x_m,x_1,\dots,x_{m - 1}),
\end{equation*}
and $Y$ and $d = (d_1,d_2,\dots,d_m) \in Y$ be as in \Lem{eLem}.
$R$ is clearly a linear isometry and $R^m = I_X$. Then $Ad =  (\tsum_{i = 1}^{m}d_i)(1,1,\dots,1)$ and so, since $d \in Y$, $\tsum_{i = 1}^{m}d_i = 0$.  Then, for all $x = (x_1,x_2, \dots, x_m) \in X$,
\begin{equation}\label{S}
S(x_1,x_2,\dots,x_m) = (x_m - x_1,x_1 - x_2, \dots, x_{m - 1} - x_m).
\end{equation}
As we postulated in \Thm{Pthm}, $P_C$ is the projection from $X$ onto $C$.    For all $x = (x_1,x_2, \dots, x_m) \in X$, $P_C(x) = (P_1x_1,P_2x_2, \dots, P_mx_m)$, and $z$ is a cycle exactly when $z = P_CRz$.   (See \cite[Equations (12), (14) and (16), p.\ 3]{ABRW}.)   We now have a simple characterization of cycles.
\begin{lemma}\label{dLem}
$(z_1,z_2,\dots,z_m)$ is a cycle if, and only if,
\begin{equation}\label{zd}
\all\ i = 1,\dots,m,\ z_i \in C_i \hbox{ and, }
\all\ i = 1,\dots, m - 1,\ z_{i} - z_{i + 1} = d_{i + 1}.
\end{equation}
\begin{proof}
Since $\tsum_{i = 1}^{m}d_i = 0$, it follows from \eqref{zd} that $z_m - z_1 = d_1$ and the result is immediate from \Thm{Pthm}(\eqref{P1}$\iff$\eqref{P5}) and \eqref{S}. 
\end{proof}
\end{lemma}
It is now convenient to make a notational change.  Let $v := -R^{m - 1}d$, from which $\tsum_{i = 1}^{m}v_i = 0$ and $d = -Rv$.   With this substitution, we get:  
\begin{theorem}\label{vThm}
$(z_1,z_2,\dots,z_m)$ is a cycle if, and only if,
\begin{equation}\label{zv}
\all\ i = 1,\dots,m, z_i \in C_i \hbox{ and, }
\all\ i = 1,\dots,m - 1,\ z_{i + 1} = z_{i} + v_{i}.
\end{equation}
\end{theorem}
We now prove the geometry conjecture:
\begin{theorem}\label{Geometric}
We have
\begin{equation}\label{w2}
\{z_m: (z_1,\dots z_m)\hbox{ is a cycle}\} =   C_{m} \cap
(C_{m - 1} + v_{m - 1}) \cap
(C_{1} + \tsum_{i = 1}^{m - 1}v_i).
\end{equation}
\end{theorem}
\begin{proof}
($\subset$)\enspace If $(z_1,z_2,\dots,z_m)$ is a cycle then, from \eqref{zv},  for all $k = 1,\ \dots,\ m - 1$, $z_{m} = z_{k} + \tsum_{i = k}^{m - 1}v_i \subset C_{k} + \tsum_{i = k}^{m - 1}v_i$.   Since $z_m \in C_m$, 
\begin{equation}\label{v}
z_m \in C_{m} \cap
(C_{m - 1} + v_{m - 1}) \cap
(C_{1} + \tsum_{i = 1}^{m - 1}v_i).
\end{equation}
Thus we have established the inclusion ($\subset$) in \eqref{w2}.  
\par
($\supset$)\enspace Suppose, conversely, that $\zeta \in C_{m} \cap
(C_{m - 1} + v_{m - 1}) \cap (C_{1} + \tsum_{i = 1}^{m - 1}v_i)$. Then there exists $w = (w_1,w_2,\dots,w_m) \in C$ such that
\begin{equation}
\zeta = w_{m} = w_{m - 1} + v_{m - 1} = \dots = w_{2} + \tsum_{i = 2}^{m - 1}v_i = w_{1} + \tsum_{i = 1}^{m - 1}v_i.
\end{equation}
From \eqref{S}, using the facts that $\tsum_{i = 1}^{m}v_i = 0$ and $d = -Rv$, 
\begin{align*}
Sw &= S(w_1,w_2,\dots,w_m)
= (w_m - w_1,w_1 - w_2, \dots, w_{m - 1} - w_m)\\
&= (\tsum_{i = 1}^{m - 1}v_i,- v_1, \dots,- v_{m - 1}) = (-v_m,- v_1, \dots,- v_{m - 1}) = -Rv = d.
\end{align*}
It now follows from \Thm{Pthm}(\eqref{P5}$\lr$\eqref{P1}) that $w = P_CRw$, i.e., $w$ is a cycle.   Since $\zeta = w_m$, we have established the inclusion ($\supset$) in \eqref{w2}.  
\end{proof}
\begin{remark}
Even if there are no cycles, \Lem{eLem} still provides $d \in Y$\break satisfying \eqref{mSe} and \eqref{mz}.   If we define $v := -R^{m - 1}d$ as above, $v$ is known as the {\em generalized gap vector}.   See \cite{ABW} for more on this. 
\end{remark}
\section{Other approaches leading to \Thm{Pthm}}\label{Other}
\Thm{FanThm} below was proved by Fan in \cite[Theorem 5, p\ 525]{FanRef}.   It is closely related to fixed point theorems, and has found many applications in optimization and mathematical economics.   It first appeared in \cite{FanOld}.
\begin{theorem}\label{FanThm}
Let $D$ be a compact, convex subset of a topological vector space, and $f\colon\ D \times D \to \RR$.   Suppose that $f$ is quasiconcave with respect to its first variable and lower semicontinuous with respect to its second variable.   Then
\begin{equation*}
\minn_{d \in D}\supn_{b \in D}f(b,d) \le \supn_{b \in D}f(b,b).
\end{equation*}
\end{theorem}
\par
It was actually proved in \eqref{bb1} and \eqref{bb2} in the proof of \Lem{mLem}, that, for all $b \in D$, $f(b,b) = 0$.   Thus, from \Thm{FanThm}, $\minn_{d \in D}\supn_{b \in D}f(b,d) \le 0$, and we obtain \eqref{mz} as before.   This is the {\em  Fan's inequality approach}.
\bigbreak
The statement of \Lem{eLem} below is identical with that of \Lem{mLem}, except that \eqref{z} and \eqref{Se} are in the reverse order to that of \eqref{mSe} and \eqref{mz}, and the quantifier ``$x \in X$'' in \eqref{z} is less restrictive than the quantifier ``$x \in S^{-1}D$'' in \eqref{mz}.   So \Lem{eLem} can be used instead of \Lem{mLem} to obtain \Thm{Pthm}.   Thus we obtain the {\em maximally monotone operator approach}.
\par
The first of the results on maximal monotonicity that we will use is\break Rockafellar's maximal monotonicity theorem, \cite[Proposition 1, pp.\ 211--212]{RTRMMT}, (see \cite[Theorem 4.6, p.\ 10]{QD2} and \cite[Theorem 3.2, pp. 634--635]{SW} for recent\break developments), which we state in \Thm{SUBDIFF} below:
\begin{theorem}\label{SUBDIFF}
Let $f$ be a proper, convex lower semicontinuous function on a Banach space.   Then the subdifferential of $f$, $\partial f$, is maximally monotone.
\end{theorem}
The second result on maximal monotonicity that we shall use follows from Rockafellar's sum theorem, \cite[Proposition 1, p.\ 77]{RTRSUM} (see \cite[Theorem 8.4\break(a)$\lr$(d), pp.\ 1036--1037]{QD1} for recent developments), which we state in\break \Thm{SUMS} below:
\begin{theorem}\label{SUMS}
Let $S$ and $T$ be maximally monotone operators on a Banach space and $S$ have full domain.   Then $S + T$ is maximally monotone.
\end{theorem}
The final result on maximal monotonicity that we shall use is Minty's\break theorem, \cite{Minty} or \cite[Theorem 21.1, pp.\ 311]{BC}, which we state in \Thm{MintyThm} below:
\begin{theorem}\label{MintyThm}
Let $S$ be a maximally monotone operator on a Hilbert space $Y$ then there exists $y \in Y$ such that $0 \in y + Sy$.
\end{theorem}  
\begin{lemma}\label{eLem} There exist $d,e \in Y$ such that,
\begin{equation}\label{z}
\all\ {x \in X},\ \sigma_{C}(Se) + \bra{Sx -  Se}{e} - \sigma_{C}(Sx) \le 0,
\end{equation}
and
\begin{equation}\label{Se}
d = Se \in D.
\end{equation}
\end{lemma}
\begin{proof}
$\sigma_C$ is proper, convex and lower semicontinuous on $Y$.   From \Thm{SUBDIFF},  $\partial\sigma_{C}$ is maximally monotone.   From \eqref{m2AB},
\begin{equation*}
\all\ y \in Y,\quad \bra{(-Q - \half I_Y)y}{y} = 0
\end{equation*}
and so $-Q - \half I_Y$ is {\em skew} with full domain $Y$ and hence maximally monotone and so, from \Thm{SUMS},  $-Q - \half I_Y + \partial\sigma_{C}$ is also maximally monotone, where $\partial\sigma_{C}$ is the subdifferential of $\sigma_C$ in the Hilbert space $Y$. 
\par
\Thm{MintyThm}, or consideration of the resolvent of $-2Q - I_Y + 2\partial\sigma_{C}$ evaluated at $0$, provides $d \in Y$ such that
\begin{equation*}
0 \in \half d + (-Q - \half I_Y + \partial\sigma_{C})(d),\hbox{ that is to say, }0 \in -Qd + \partial\sigma_{C}(d).
\end{equation*}
Thus
\begin{equation*}
Qd \in \partial\sigma_{C}(d)\dand\supn_{y \in Y}[\sigma_{C}(d) + \bra{y -  d}{Qd} - \sigma_{C}(y)] \le 0.
\end{equation*}
If $x \in X$ then, from \eqref{SXY}, $Sx \in Y$ and so 
\begin{equation}\label{ySY}
\supn_{x \in X}[\sigma_{C}(d) + \bra{Sx  -  d}{Qd} - \sigma_{C}(Sx)] \le 0.
\end{equation}
Let $e := Qd \in Y$.   (Compare \cite[Equation (40), p.\ 6]{ABRW}.)    From  \eqref{SB}, $Se = d \in Y$, so \eqref{z} follows from \eqref{ySY}.  If we substitute $x = 0$ in \eqref{z}, we see that\break $\sigma_{C}(Se) \le \bra{Se}{e}$.\quad  From \eqref{3Norms},\quad $\bra{Se}{e} = -\half\|Se\|^2$.\quad   Thus\quad $\sigma_{C}(Se) \le -\half\|Se\|^2$,\quad and \eqref{Se} follows.
\end{proof}


\begin{thebibliography}{50}

\bibitem{ABRW2}Salihah Alwadani, Heinz H. Bauschke, Julian P. Revalski and Xianfu Wang, {\em Resolvents and Yosida approximations of displacement mappings of isometries}, Set-Valued and Variational Analysis {\bf 29} (2021), pp. 721-733, 

\bibitem{ABW}Salihah Alwadani, Heinz H. Bauschke, and Xianfu Wang, {\em Attouch-Th\'era duality, generalized cycles and gap vectors}, SIAM Journal on Optimization {\bf 31} (2021), pp.\ 1926-1946.

\bibitem{ABRW}Salihah Alwadani, Heinz H. Bauschke, Julian P. Revalski and Xianfu Wang, {\em The difference vectors for convex sets and a resolution of the geometry conjecture},
Open Journal of Mathematical Optimization, Volume 2 (2021), Article 5 (18 pages).    DOI: 10.5802/ojmo.7
%
\bibitem{BB} Heinz H. Bauschke, Jonathan M. Borwein, and Adrian S. Lewis, {\em The method of cyclic projections for closed convex sets in Hilbert space}. In\break Recent developments in optimization theory and nonlinear analysis, Contemporary Mathematics {\bf 104} (1997), pp.\ 1--38. American Mathematical\break Society.
%
\bibitem{BC}Heinz H. Bauschke and Patrick L. Combettes, {\em Convex analysis and monotone operator theory in Hilbert spaces}.   CMS Books in Mathematics/Ouvrages de Math\'ematiques de la SMC. Springer, second edition, 2017.
%
\bibitem{FanRef}Ky Fan, {\em Some properties of convex sets related to fixed point theorems}, Math. Ann. {\bf 266} (1984), 519--537.
%
\bibitem{FanOld}Ky Fan, {\em A minimax inequality and application}, Inequalities, III (Proc. Third Sympos., Univ. California, Los Angeles, Calif., 1969; dedicated to the memory of Theodore S. Motzkin), pp. 103--113. Academic Press, New York, 1972. 
%
\bibitem{Minty}George J. Minty, {\em Monotone (nonlinear) operators in Hilbert space}, Duke Math. J. 29 (1962), 341-–346.
%
\bibitem{RTRMMT}R. Tyrrell Rockafellar, {\em On the maximal monotonicity of subdifferential\break mappings}, Pacific J. Math {\bf 33} (1970), 209--216.

\bibitem{RTRSUM}R. Tyrrell Rockafellar, {\em On the Maximality of
Sums of Nonlinear Monotone Operators}, Trans. Amer. Math. Soc. {\bf 149} (1970), 75--88.

\bibitem{SURVEY}Stephen Simons, {\em Minimax theorems and their proofs},
Minimax and applications, Ding-Zhu Du and Panos M. Pardalos eds.,
Kluwer Academic Publishers, Dordrecht -- Boston (1995), 1--23.

\bibitem{MANDM}Stephen Simons, {\em Minimax and monotonicity}, Lecture Notes in Mathematics {\bf 1693} (1998), Springer–Verlag.

\bibitem{HBM}Stephen Simons, {\em From Hahn--Banach to monotonicity}, 
Lecture Notes in Mathematics, {\bf 1693}, second edition, (2008), Springer--Verlag.

\bibitem{QD1}Stephen Simons, {\em ``Densities'' and maximal monotonicity}, J. of Convex Anal. {\bf 23} (2016), 1017--1050.

\bibitem{QD2}Stephen Simons, {\em Quasidense Monotone Multifunctions}, Set-Valued  Var. Anal. {\bf 26} (2018), pp. 5--26.  DOI: 10.1007/s11228-017-0434-7.

\bibitem{SW}Stephen Simons and Xianfu Wang, {\em Ubiquitous subdifferentials, $r_L$-density and maximal monotonicity},  Set-Valued Var. Anal. {\bf 23} (2014), 631--642.\break  DOI:   10.1007/s11228-015-0326-7.

\end{thebibliography}
\end{document}